\documentclass[12pt]{article} 
\usepackage[cp1251]{inputenc}
\usepackage{graphicx}
\usepackage{cite}
\usepackage{textcase}
\usepackage{amsfonts}
\usepackage{amsmath}
\usepackage[russian]{babel}

\newtheorem{theorem}{Теорема}
\newtheorem{remark}{Замечание}
\newtheorem{proof}{Доказательство}

\begin{document} 

\begin{center}
\textbf{Управление динамическими системами при ограничениях на входные и выходные сигналы}
\end{center}

\begin{center}
 Игорь Борисович~Фуртат
 \\
 Гущин Павел Александрович
 \\
 Нгуен Ба Хю
\end{center}

\begin{center}
Институт проблем машиноведения РАН, 2021 г.
\end{center}

\begin{abstract}

В статье рассмотрено развитие метода, предложенного в публикации И.Б. Фуртата, П.А. Гущина,  <<Автоматика и телемеханика>>, 2021, № 4, на системы с произвольным соотношением количества входных и выходных сигналов и гарантией нахождения их в заданном множестве. 
Для решения задачи предложены две последовательные замены координат. 
Первая замена сводит выходную переменную объекта к новой переменной, размерность которой не превосходит размерности вектора управления. 
Вторая замена позволяет перейти от задачи управления с ограничениями к задаче управления без ограничений. 
В качетсве иллюстрации работоспособности метода рассмотрено решение двух задач. 
Первая задача -- управление по состоянию линейными системами с учетом ограничений на сигнал управления и фазовые переменные. 
Вторая задача -- управление по выходу линейными системами с ограничением на выходной сигнал и сигнал управления. 
В обеих задачах проверка устойчивости замкнутой системы формулируется в терминах разрешимости линейных матричных неравенств.
Полученные результаты сопровождаются примерами моделирования, иллюстрирующими эффективность предложенного метода.
\end{abstract}

\textit{Ключевые слова:} динамическая система, замена координат, устойчивость, управление.

\section{Введение}

В \cite{Furtat21} приведен обзор методов управления с обеспечением выходного сигнала объекта в заданном множестве.
Также в \cite{Furtat21} предложено новое решение данной задачи, которое состоит в использовании специальной замены координат, позволяющей перейти от задачи с ограничениями к задаче без ограничений.

Однако и решение \cite{Furtat21} и большая часть описанной литературы в \cite{Furtat21} ограниченны рассмотрением объектов, у которых: 
\begin{enumerate}
\item [1)] размеренность сигнала управления больше или равна размерности регулируемого сигнала;
\item [2)] не накладываются ограничения на сигнал управления.
\end{enumerate}

Касательно первого ограничения на практике существует достаточно много задач, когда размерность сигнала управления меньше размерности регулируемого сигнала. 
Например, управление неполноприводными системами: шагающие роботы, станки с программным управлением, летательные аппараты, водные суда, некторые маятниковые системы и т.д. 
Для решения таких задач предложено много методов \cite{Miroshnik00,Spong01,Olfati01,Ortega02,Grishin02,Andrievski04,Astrom07,Liu13,Liu20,Sun20,Saleem21}
, однако они не гарантируют заданного качетсва регулирования в любой момент времени, а лишь в установившемся режиме. 

Относительно второго ограничения логически возникает вопрос о величине сигнала управления необходимого для обеспечения выходного сигнала в заданном множестве в любой момент времени.
 
В данной работе подход \cite{Furtat21} будет развит для решения следующих задач: 
\begin{enumerate}
\item [1)] управление объектами, у котрых размерность сигнала управления может быть меньше размерности регулируемого сигнала;
\item [2)] обеспечение регулируемого и управляющего сигналов в заданном множестве в любой момент времени;
\item [3)] использование аппарата линейных матричных неравенств для анализа устойчивости замкнутой системы и синтеза параметров регулятора.
\end{enumerate}

Статья организована следующим образом. 
В разделе \ref{Sec2} ставится общая задача управления с гарантией нахождения входного и выходного сигналов в заданном множестве в любой момент времени. 
В разделе \ref{Sec3} предложены две замены координат. 
Первая замена сводит выходную переменную объекта к новой переменной, размерность которой не превосходит размерности вектора управления. 
Вторая замена позволяет перейти от задачи управления с ограничениями к задаче управления без ограничений. 
В разделе \ref{Sec4}  рассмотрено применение результата из раздела \ref{Sec3} к задаче управления по состоянию линейными системами с учетом ограничений на сигнал управления и фазовые переменные. 
В разделе \ref{Sec5}  решена задача управления по выходу линейными системами с ограничением на выходной и управляющий сигналы. 
Полученные результаты сопровождаются примерами моделирования, иллюстрирующими эффективность предложенного метода.

В статье используются следующие обозначения:
$\mathbb R^{n}$ -- евклидово пространство размерности $n$ с нормой $|\cdot|$; 
$\mathbb R^{n \times m}$ -- множество всех $n \times m$ вещественных матриц; 
$P>0$ -- положительно определенная симметричная матрица.

\section{Постановка задачи}
\label{Sec2}
Рассмотрим динамическую систему вида
\begin{equation}
\label{eq2_1}
\begin{array}{l} 
\dot{x}=F(x,u,t),
\\
y=H(x,u,t),
\end{array}
\end{equation} 
где $t \geq 0$, $x \in \mathbb R^n$ -- вектор состояния, 
$u \in \mathcal U \subset \mathbb R^m$ -- сигнал управления, 
$y=col\{y_1,...,y_l\} \in \mathcal Y \subset \mathbb R^{v}$ -- выходной сигнал, 
функции $F$ и $H$ определены для всех $x$, $u$ и $t$, 
$F$ кусочно-непрерывна и ограничена по $t$, 
функция $H$ непрерывно дифференцируемая по всем аргументам и ограниченная по $t$. Объект управления \eqref{eq2_1} стабилизируемый и наблюдаемый для любого $x \in \mathbb R^n$. 

В отличие от \cite{Furtat21}, в настоящей статье не требуются ограничения вида $\dim u \geq \dim y$, а возможно произвольное соотношение между размерностями сигналов $u$ и $y$.
Для решения данной задачи введем новую переменную
\begin{equation}
\label{eq2_2}
\xi=G(y,u,t)
\end{equation}
размерность которой не превосходит размерности управляющего сигнала, т.е. $\xi = col\{\xi_1,...,\xi_v\}$, где $v \leq m$, 
$G$ -- непрерывно дифференцируемая функция по всем аргументам и ограниченная по $t$.  

Требуется разработать закон управления, который обеспечит нахождение нового выходного сигнала $\xi(t)$ в следующем множестве
\begin{equation}
\label{eq2_3}
\begin{array}{l} 
\mathcal{S}=\left\{\xi \in \mathbb R^v:~ \underline{g}_i(t)<\xi_i(t)<\overline{g}_i(t),~i=1,...,v \right\}.
\end{array}
\end{equation} 
Функция $G$ и непрерывно-дифференцируемые функции $\underline{g}_i(t)$ и $\overline{g}_i(t)$ выбираются так, чтобы $\mathcal S \subseteq \mathcal U$ для любых $y \in \mathcal Y$, $t \geq 0$ и $\mathcal S \subseteq \mathcal Y$ для любых $u \in \mathcal U$, $t \geq 0$. 
Т.о. ограничения \eqref{eq2_3} вместе с преобразованием \eqref{eq2_2} включают в себя ограничения на входные и выходные сигналы объекта \eqref{eq2_1}. 
Дополнительно, выбором функций $\underline{g}_i(t)$ и $\overline{g}_i(t)$ можно задавать различную <<конфигурацию>> множества, в котором осуществляется переходной процесс $\xi_i(t)$.


\section{Метод решения}
\label{Sec3}
Следуя \cite{Furtat21}, введем замену переменной $\xi$ в виде
\begin{equation}
\label{eq3_1}
\begin{array}{l} 
\xi(t)=\Phi(\varepsilon(t),t),
\end{array}
\end{equation} 
где $\varepsilon(t) \in \mathbb R^v$ -- непрерывно-дифференцируемая функция по $t$, 
 $\Phi(\varepsilon,t)=col\{\Phi_1(\varepsilon,t),...,\Phi_v(\varepsilon,t)\}$ удовлетворяет следующим условиям:
\begin{enumerate}

\item [(а)] $\underline{g}_i(t) < \Phi_i(\varepsilon,t) 
 < \overline{g}_i(t)$, $i=1,...,v$ для любых $t \geq 0$ и $\varepsilon \in \mathbb R^v$;

\item [(б)] существует обратное отображение $\varepsilon=\Phi^{-1}(\xi,t)$ для любых $\xi \in \mathcal{S}$ и $t \geq 0$;

\item [(в)] функция $\Phi(\varepsilon,t)$ непрерывно-дифференцируемая по $\varepsilon$ и $t$, а также $\det\left(\frac{\partial \Phi(\varepsilon,t)}{\partial \varepsilon}\right) \neq 0$ для любых $\varepsilon \in \mathbb R^v$ и $t \geq 0$;

\item [(г)] $\left| \frac{\partial \Phi(\varepsilon,t)}{\partial t} \right| \leq \gamma$, $\gamma>0$ для любых $\varepsilon \in \mathbb R^v$ и $t \geq 0$.
\end{enumerate}

Для синтеза закона управления потребуется знание динамики переменной $\varepsilon(t)$. 
Для этого найдем полные производные по времени от $y(t)$ и $\xi(t)$ вдоль траекторий \eqref{eq2_1}, \eqref{eq2_2} и \eqref{eq3_1}:
\begin{equation}
\label{eq3_2}
\begin{array}{l}
\dot{y}=\frac{\partial H}{\partial x}F
+\frac{\partial H}{\partial u}\dot{u}
+\frac{\partial H}{\partial t},
\\
\dot{\xi}=\frac{\partial\Phi(\varepsilon,t)}{\partial \varepsilon}\dot{\varepsilon}
+\frac{\partial\Phi(\varepsilon,t)}{\partial t},
\\
\dot{\xi}=\frac{\partial G}{\partial y}\dot{y}
+\frac{\partial G}{\partial u}\dot{u}
+\frac{\partial G}{\partial t}.
\end{array}
\end{equation}
Принимая во внимание условие (в), 
 выразим $\dot{\varepsilon}$ из \eqref{eq3_2} в виде 
\begin{equation}
\label{eq3_3}
\begin{array}{l} 
\dot{\varepsilon}=\left(\frac{\partial\Phi(\varepsilon,t)}{\partial \varepsilon}\right)^{-1}
\Big[
\frac{\partial G}{\partial y} \frac{\partial H}{\partial x}F
+
\Big(
\frac{\partial G}{\partial y}\frac{\partial H}{\partial u}+\frac{\partial G}{\partial u}
\Big) \dot{u}
+\frac{\partial G}{\partial y}\frac{\partial H}{\partial t}
+\frac{\partial G}{\partial t}
-\frac{\partial\Phi(\varepsilon,t)}{\partial t}
\Big].
\end{array}
\end{equation}

\begin{remark}
Если модель объекта задана в форме передаточной функции, то для синтеза закона управления вместо \eqref{eq3_3} удобнее использовать выражение, полученное из последних двух уравнений в \eqref{eq3_2} в виде
\begin{equation}
\label{eq3_03}
\begin{array}{l} 
\dot{\varepsilon}=\left(\frac{\partial\Phi(\varepsilon,t)}{\partial \varepsilon}\right)^{-1}
\Big[
\frac{\partial G}{\partial y} \dot{y}
+
\frac{\partial G}{\partial u} \dot{u}
+\frac{\partial G}{\partial t}
-\frac{\partial\Phi(\varepsilon,t)}{\partial t}
\Big].
\end{array}
\end{equation}
\end{remark}

В следующих разделах выражения \eqref{eq3_3} и \eqref{eq3_03} будут использоваться для синтеза законов управления. 
Теперь сформулируем основной результат настоящего раздела.

\begin{theorem}
\label{Th3_1}
Пусть для преобразования \eqref{eq3_1} выполнены условия (а)-(г). 
Если выбранный закон управления $u \in \mathcal U$ обеспечивает ограниченность решений \eqref{eq3_3}, 
то будет выполнено целевое условие \eqref{eq2_3}. 
\end{theorem}

Доказателсьтво теоремы \ref{Th3_1} следует непосредственно из доказательства теоремы 1 в \cite{Furtat21}.
В результате теорема \ref{Th3_1} позволяет перейти от задачи управления \eqref{eq2_1} с ограничениями \eqref{eq2_2} к задаче управления \eqref{eq2_3} без ограничений.

В \cite{Furtat21} предложены различные частные виды замены координат \eqref{eq3_1}, не связанные друг с другом. 
Ниже приведем новый вид замены координат, которая позволит связать преобразования из \cite{Furtat21} и получить ряд новых видов замен.

\sl{П\,р\,и\,м\,е\,р\,\,1}.
\rm
Пусть $\Phi(\varepsilon,t) \in \mathbb R$ в \eqref{eq3_1} задана в виде
\begin{equation}
\label{eqC3_2_0004}
\begin{array}{l}
\Phi(\varepsilon,t)=\frac{\overline{g}(t)-\underline{g}(t)}{2}T(\varepsilon)+\frac{\underline{g}(t)+\overline{g}(t)}{2}, 
\end{array}
\end{equation}
где $\varepsilon \in \mathbb R$,
$T(\varepsilon)$ -- строго монотонная функция такая, что 
$-1 < T(\varepsilon) < 1$ для любых $\varepsilon$. 
 
Достоинство замены координат \eqref{eqC3_2_0004}, по сравнению с \cite{Furtat21}, состоит в том, что она позволяет отдельно выделить 
функции $\overline{g}(t)$, $\underline{g}(t)$ и $T(\varepsilon)$.  
Функции  $\overline{g}(t)$ и $\underline{g}(t)$ определяют желаемое множество для регулируемой переменной и задаются разработчиком.
Функция $T(\varepsilon)$ определяет замену координат. 
Например, $T(\varepsilon)$ может быть выбрана в виде
$T(\varepsilon)=\frac{\varepsilon}{1+|\varepsilon|}$, $T(\varepsilon)=\frac{e^\varepsilon-1}{e^\varepsilon+1}=th(0,5\varepsilon)$,
 $T(\varepsilon)=\frac{2}{\pi} arctg(\varepsilon)$ и т.п.
 
Ниже рассмотрим применение предложенного подхода для тех же типов моделей, что и в \cite{Furtat21}. 
Обобщение полученных результатов на случай неизвестных параметров модели и дополнительных нелинейных слагаемых в ней может быть непосредственно получена, как в \cite{Furtat21_IJC}.


\section{Управление по состоянию с ограничением на фазовые переменные и сигнал управления}
\label{Sec4}

Рассмотрим объект управления вида
\begin{equation}
\label{eq4_1}
\begin{array}{l}
\dot{x}=Ax+Bu+Df,
\end{array}
\end{equation}
где $x \in \mathcal X \subset \mathbb R^n$, $u \in \mathcal U \subset \mathbb R$, $f \in \mathbb R$ и $|f(t)| \leq \bar{f}$ для всех $t$, матрицы $A$, $B$ и $D$ известны и имеют соответствующие размерности. 
Множества $\mathcal X$ и $\mathcal U$ заданы как
\begin{equation}
\label{eq_XU_4}
\begin{array}{l}
\mathcal X = \left\{x \in \mathbb R^n:~ x^{\rm T}P_1x \leq l_x(t) \right\},
~~~
\mathcal U = \left\{u \in \mathbb R:~ p_0 u^2 \leq l_u(t) \right\},
\end{array}
\end{equation}
где $P_1>0$, $p_0>0$, $l_x(t) > 0$ и $l_u(t)>0$ задаются разработчиком.

Представим сигнал управления в виде следующей суммы:
\begin{equation}
\label{eq_C_L}
u=u_1+u_2,
\end{equation}
где $u_1$ будет использоваться для стабилизации \eqref{eq4_1}, $u_2$ -- для обеспечения заданных ограничений на $x$ и $u$.

С учетом неравенством Юнга 
($2a^{\rm T}b \leq \mu a^{\rm T}a + \mu^{-1}b^{\rm T}b$ для любых $a,b \in \mathbb R^n$ и $\mu>0$), имеем $u^2= u_1^2+2u_1u_2+u_2^2 \leq (1+r)u_1^2+(1+r^{-1})u_2^2$, $r>0$. 
Обозначив $p_2=p_0(1+r)$ и $p_3=p_0(1+r^{-1})$, определим переменную $\xi$ в \eqref{eq2_2} в виде
\begin{equation}
\label{eq4_2}
\xi=x^{\rm T} P_1 x + p_2 u_1^2 + p_{3}(|u_2|+\delta)^2,
\end{equation}
где величина $\delta>0$ задается разработчиком и потребуется для реализации закона управления $u_2$ с учетом того, что $|u_2|+\delta \neq 0$. 

Так как $\xi$ определена скалярной величиной, то перепишем целевое условие \eqref{eq2_3} как
\begin{equation}
\label{eq3_3_G}
\begin{array}{l} 
\mathcal{S}=\left\{\xi \in \mathbb R:~ \underline{g}(t)<\xi(t)<\overline{g}(t) \right\}.
\end{array}
\end{equation} 
Выбрав $\overline{g}(t)$ в виде $\overline{g}(t) \leq \inf\{l_x(t),l_u(t)\}$ получим, что множество $\mathcal{S}$ включает в себя ограничения \eqref{eq_XU_4}.

Цель управления состоит в выполнении целевого условия \eqref{eq3_3_G}. 
Таким образом фазовые переменные и сигнал управления будут находиться в заданных множествах $\mathcal X$ и $\mathcal U$ с дополнительным выполнением \eqref{eq3_3_G}. 

Отметим, что метод \cite{Furtat21} для объекта \eqref{eq4_1} позволяет решить только задачу нахождения одной  компоненты вектора $x$ в заданном множестве без ограничений на сигнал управления.

Сформулируем основной результат раздела \ref{Sec4}.

\begin{theorem}
\label{Th02}
Пусть для преобразования \eqref{eq2_2} выполнены условия (а)-(г), $\frac{\partial\Phi(\varepsilon,t)}{\partial \varepsilon}>0$  для любых $\varepsilon$ и $t$. 
 Если для заданных $\delta>0$, $\mu>0$, $P_1>0$, $p_2>0$, $p_3>0$, $\beta>0$, $c>0$ и $K \in \mathbb R^{1 \times n}$ существуют $H>0$, $\alpha>0$, $\tau_i>0$, $i=1,...,5$ такие, что при $v=\pm \bar{f}$ разрешимы следующие линейные неравенства:
\begin{equation}
\label{eq_LMI01}
\begin{array}{l}
\begin{bmatrix}
-\alpha + 0.5 \tau_1  & 0.5 v \mu^{-1} D^{\rm T} D & -0.5 \\
* & -\tau_2 & 0 \\
* & * & -\tau_3
\end{bmatrix} \leq 0,
\\
c \tau_1 \geq \bar{f}^2 \tau_2 + \gamma^2 \tau_3,
\end{array}
\end{equation}
\begin{equation}
\label{eq_LMI02}
\begin{array}{l}
\begin{bmatrix}
\bar{A}^{\rm T}H+H\bar{A}+\beta H  & H B & H D \\
* & -\tau_4 & 0 \\
* & * & -\tau_5
\end{bmatrix} \leq 0,
\\
H \geq \bar{P}_1,
\\
\frac{\inf\{\bar{g}(t)\}}{\lambda_{\min}\{\bar{P}_1\}} \beta \geq \frac{\inf\{\bar{g}(t)\}}{p_3} \tau_4 + \bar{f}^2 \tau_5,
\end{array}
\end{equation}
то закон управления \eqref{eq_C_L} обеспечит выполнение целевого условия \eqref{eq3_3_G}, где 
\begin{equation}
\label{eq4_3}
\begin{array}{l}
u_1=Kx,
\\
\dot{u}_2=-\frac{2}{p_3(|u_2|+\delta)} \textup{sign} (u_2) 
\left[
\alpha \varepsilon  
+2x^{\rm T} \bar{P}_1 \bar{A}x
+2 x^{\rm T} \bar{P}_1 B u_2
+\mu \textup{sign} (\varepsilon) x^{\rm T} \bar{P}_1^2 x
\right],
\end{array}
\end{equation} $\bar{A}=A+BK$ и 
$\bar{P}_1=P_1 + p_2 K^{\rm T}K$.
\end{theorem}

\begin{proof}
Принимая во внимание \eqref{eq_C_L} и \eqref{eq4_3}, преобразуем \eqref{eq2_1} и \eqref{eq4_2} к виду
\begin{equation}
\label{eq4_002}
\begin{array}{l}
\dot{x}=\bar{A} x + Bu_2 + Df,
\\
\xi=x^{\rm T} \bar{P}_1 x + p_{3}(|u_2|+\delta)^2.
\end{array}
\end{equation}
С учетом \eqref{eq4_002} перепишем \eqref{eq3_3} как
\begin{equation}
\label{eq4_02}
\begin{array}{l}
\dot{\varepsilon}=\left(\frac{\partial\Phi(\varepsilon,t)}{\partial \varepsilon}\right)^{-1}
\Big[
2x^{\rm T}\bar{P}_1 \bar{A}x
+2 x^{\rm T} \bar{P}_1 B u
+2x^{\rm T}\bar{P}_1Df +
\\
~~~~~~
+2 p_{3}(|u_2|+\delta) \textup{sign}(u_2)  \dot{u}_2
-\frac{\partial\Phi(\varepsilon,t)}{\partial t}
\Big].
\end{array}
\end{equation}

Для анализа устойчивости решений \eqref{eq4_02} рассмотрим функцию Ляпунова вида
\begin{equation}
\label{eq4_6}
\begin{array}{l}
V_1=0.5 \varepsilon^{2}.
\end{array}
\end{equation}

Найдем полную производную по времени от \eqref{eq4_6} вдоль решений \eqref{eq4_02} и затем воспользуемся неравенством Юнга. 
В результате получим:
\begin{equation}
\label{eq4_6_dot_V}
\begin{array}{lll}
\dot{V}_1
= &
\left(\frac{\partial\Phi(\varepsilon,t)}{\partial \varepsilon}\right)^{-1}
\varepsilon
\Big[
2x^{\rm T}\bar{P}_1 \bar{A} x +2 x^{\rm T} \bar{P}_1 B u
+2x^{\rm T}Df +
\\
&
+2 p_{3}(|u_2|+\delta) \textup{sign}(u_2)  \dot{u}_2
-\frac{\partial\Phi(\varepsilon,t)}{\partial t}
\Big] \leq
\\
 & \leq
\left(\frac{\partial\Phi(\varepsilon,t)}{\partial \varepsilon}\right)^{-1}
\varepsilon
\Big[
2x^{\rm T}\bar{P}_1 \bar{A}x +2 x^{\rm T} \bar{P}_1 B u
+\mu \textup{sign}(\varepsilon) x^{\rm T} \bar{P}_1^2 x  +
\\
&
+2 p_{3}(|u_2|+\delta) \textup{sign}(u_2)  \dot{u}_2
+ \mu^{-1} \textup{sign}(\varepsilon) D^{\rm T} \bar{P}_1^2 D f^2 -\frac{\partial\Phi(\varepsilon,t)}{\partial t}
\Big].
\end{array}
\end{equation}

С учетом второго выражения в \eqref{eq4_3}, перепишем \eqref{eq4_6_dot_V} как

\begin{equation}
\label{eq4_6_dot_V1}
\begin{array}{l}
\dot{V}_1=
\left(
\frac{\partial\Phi(\varepsilon,t)}{\partial \varepsilon}\right)^{-1}
\left( - \alpha \varepsilon^2 
+ \mu^{-1} \textup{sign}(\varepsilon) \varepsilon D^{\rm T} D f^2 
-\varepsilon \frac{\partial\Phi(\varepsilon,t)}{\partial t} 
\right).
\end{array}
\end{equation}

Потребуем при $V_1 \geq c$ выполнение условия $\dot{V}_1 \leq 0$, принимая во внимание ограничения $f^2 \leq \bar{f}^2$ и $\left( \frac{\partial\Phi(\varepsilon,t)}{\partial t} \right)^2 \leq \gamma^2$ (см. постановку задачи и условие (г)).
 Так как $\frac{\partial\Phi(\varepsilon,t)}{\partial \varepsilon} > 0$ не влияет на знак выражения \eqref{eq4_6_dot_V1}, то перепишем вышеназванные условия в виде
\begin{equation}
\label{eq4_6_dot_V2}
\begin{array}{l}
- \alpha \varepsilon^2 
+ \mu^{-1} \varepsilon \textup{sign}(\varepsilon) D^{\rm T} D f^2
-\varepsilon \frac{\partial\Phi(\varepsilon,t)}{\partial t} 
\leq 0
~~~ \forall \left(\varepsilon, f, \frac{\partial\Phi(\varepsilon,t)}{\partial t} \right):
\\
0.5 \varepsilon^2 \geq c,
~~~f^2 \leq \bar{f}^2,
~~~\left( \frac{\partial\Phi(\varepsilon,t)}{\partial t} \right)^2 \leq \gamma^2.
\end{array}
\end{equation}

Обозначив $z=col \left \{\varepsilon,f,\frac{\partial\Phi(\varepsilon,t)}{\partial t} \right \}$, 
перепишем \eqref{eq4_6_dot_V2} в матричном виде:
\begin{equation}
\label{eq4_6_dot_V3}
\begin{array}{l}
z^{\rm T}
\begin{bmatrix}
-\alpha  & 0.5 \mu^{-1} \textup{sign}(\varepsilon) f D^{\rm T} D & -0.5 \\
* & 0 & 0 \\
* & * & 0
\end{bmatrix}
z \leq 0,
\\
z^{\rm T}
\begin{bmatrix}
-0.5  & 0 & 0 \\
* & 0 & 0 \\
* & * & 0
\end{bmatrix}
z \leq -c,
~~~
z^{\rm T}
\begin{bmatrix}
0  & 0 & 0 \\
* & 1 & 0 \\
* & * & 0
\end{bmatrix}
z \leq \bar{f}^2,
~~~

z^{\rm T}
\begin{bmatrix}
0  & 0 & 0 \\
* & 0 & 0 \\
* & * & 1
\end{bmatrix}
z \leq \gamma^2.
\end{array}
\end{equation}
Используя S-процедуру и условия разрешимости линейных матричных неравенств с политопной неопределенностью  $\textup{sign}(\varepsilon) f \in [-\bar{f}, \bar{f}]$ \cite{Fridman10,Polyak14}, выражения \eqref{eq4_6_dot_V3} будут выполнены одновременно, если будут выполнены условия \eqref{eq_LMI01}. 
Тогда $\varepsilon(t)$ будет ограниченной функцией для любых $t$. 
Из теоремы \ref{Th3_1} следует, что целевое условие \eqref{eq3_3_G} будет выполнено. 
Значит, сигналы $x$, $u_1$, $u_2$ и $u$ будут ограничены. 

Однако условие \eqref{eq3_3_G} может быть выполнено при $u$ стремящимся в пределе к нулю и $x$ стремящимся в пределе к границе эллипсоида $x^{\rm T} \bar{P}_1x=\bar{g}(t)$. 
Несмотря на то, что это удовлетворяет поставленной цели \eqref{eq3_3_G}, на практике такая ситуация может потребовать больших вычислительных затрат. 
Для избежания данной ситуации найдем дополнительные условия с использованием функции Ляпунова вида
\begin{equation}
\label{eq4_V2_1}
\begin{array}{l}
V_2=x^{\rm T} H x,
\end{array}
\end{equation}
где потребуем при $V_2 \geq \frac{\inf{\bar{g}(t)}}{\lambda_{\min}\{\bar{P}_1\}}$ выполнение условий $\dot{V} \leq 0$ и $H \geq \bar{P}_1$ (последнее означает, что эллипсоид $x^{\rm T} Q x =1$ содержится внутри эллипсоида $x^{\rm T} \bar{P}_1x =1$), принимая во внимание ограничения $u_2^2 \leq \frac{\inf{\bar{g}(t)}}{p_3}$ и $f^2 \leq \bar{f}^2$. 
Т.е. потребуем, чтобы закон управления $u_1$ гарантировал нахождение фазовых траекторий в меньшем множестве по сравнению с \eqref{eq3_3_G}.
 Перепишем данные условия как
\begin{equation}
\label{eq4_V2_dot_1}
\begin{array}{l}
\dot{V}_2=
x^{\rm T}(\bar{A}^{\rm T} H+H\bar{A})x+2x^{\rm T} H B u_2+2x^{\rm T}H Df \leq 0
~~~ \forall (x, u_2, f):
\\
x^{\rm T}Hx \geq \frac{\inf{\bar{g}(t)}}{\lambda_{\min}\{\bar{P}_1\}},
~~~u_2^2 \leq \frac{\inf{\bar{g}(t)}}{p_3},
~~~f^2 \leq \bar{f}^2.
\end{array}
\end{equation}
Обозначив $s=col\{x,u_2,f\}$, преобразуем \eqref{eq4_V2_dot_1} к виду
\begin{equation}
\label{eq4_V2_dot_2}
\begin{array}{l}
s^{\rm T}
\begin{bmatrix}
\bar{A}^{\rm T}H+H\bar{A}  & H B & H D \\
* & 0 & 0 \\
* & * & 0
\end{bmatrix}s \leq 0,
\\
- s^{\rm T}
\begin{bmatrix}
H  & 0 & 0 \\
* & 0 & 0 \\
* & * & 0
\end{bmatrix}
s \geq -\frac{\inf{\bar{g}(t)}}{\lambda_{\min}\{\bar{P}_1\}},
\\
s^{\rm T}
\begin{bmatrix}
0  & 0 & 0 \\
* & -1 & 0 \\
* & * & 0
\end{bmatrix}
s \leq \frac{\inf\{\bar{g}(t)\}}{p_3},
~~
s^{\rm T}
\begin{bmatrix}
0  & 0 & 0 \\
* & 0 & 0 \\
* & * & -1
\end{bmatrix}
s \leq \bar{f}.
\end{array}
\end{equation}
Используя S-процедуру \cite{Fridman10}, неравенства \eqref{eq4_V2_dot_2} будут выполнены одновременно, если будут выполнены условия \eqref{eq_LMI02}.
Теорема \ref{Th02} доказана.
\end{proof}



\sl{П\,р\,и\,м\,е\,р\,\,2}.
\rm
Рассмотрим неустойчивый объект \eqref{eq4_1} со следующими параметрами:
\begin{equation}
\label{eq4_9}
\begin{array}{l}
A=
\begin{bmatrix}
0 & 1 \\
1 & 2
\end{bmatrix},
~~~
B=
\begin{bmatrix}
0  \\
1
\end{bmatrix},
~~~
D=
\begin{bmatrix}
0.1 \\
1
\end{bmatrix},
\\
x(0)=
\begin{bmatrix}
-1  \\
1
\end{bmatrix},
~~~
f(t)=
0.1[\textup{sign}(\sin(1.7t))+0.2\sin(0.3t)+\textup{sat}\{d(t)\}],

\end{array}
\end{equation}
где $\textup{sat}\{\cdot\}$ -- функция насыщения, $d(t)$ -- сигнал, имитирующий белый шум с ограниченной полосой пропускания и моделируемый в Matlab Simulink с помощью блока <<Band-Limited White Noise>> со следующими параметрами: мощность шума $0.3$ и время выборки $0.2$ соответственно. 
Тогда $\bar{f}=0.22$. 

Пусть в \eqref{eq_XU_4} определены следующие параметры: $P_1=
0.1\begin{bmatrix}
1 & 0 \\
0 & 1
\end{bmatrix}$, $p_0=0.1$, $l_x(t)=e^{-0.05t}+0.5$ и $l_u(t)=1$.

Зададим 
$\delta=0.01$ в \eqref{eq4_2}, 
 $K=[-2~ -4]$ и $\mu=0.01$ в \eqref{eq4_3}, 
 а также $T(\varepsilon)=\frac{e^\varepsilon-1}{e^\varepsilon+1}$, $\overline{g}=0.9 e^{-0.1t}+0.1$ и $\underline{g}=0.01$ в \eqref{eq3_3_G}. 
 С учетом $p_0=0.1$ имеем, что $p_{2}= 0.11$ и $p_3=1.1$.
 Вычислим $\gamma=\frac{3\overline{g}-\underline{g}}{2}=1.475$. Задав $c=1$ и $\beta=0.1$, неравенства \eqref{eq_LMI01}, \eqref{eq_LMI02} имеют решения, например, при $\alpha=11.6$.
 
На рис.~\ref{Fig_1_Ex2} и \ref{Fig_Xi_Ex2} изображены фазовые портреты по $(x_1,x_2)$ и $(u_1,u_2)$, а также переходной процесс по $\xi(t)$. 
На рис.~\ref{Fig_1_Ex2} большие <<эллипсы>> соответствуют выражениям $x^{\rm T}P_1x = \overline{g}(0)$ и $p_2 u_1^2+p_3 (|u_2|+\delta)^2 = \overline{g}(0)$, а  меньшие <<эллипсы>> -- $x^{\rm T}\bar{P}_1x = \inf\{\overline{g}(t)\}$ и $p_2 u_1^2+p_3 (|u_2|+\delta)^2 = \inf\{\overline{g}(t)\}$.
Из рис.~\ref{Fig_1_Ex2} видно, что фазовые траектории начинаются в большом <<эллипсе>>, с течением времени (примерно через $60$ с после начало работы системы из рис.~\ref{Fig_Xi_Ex2}) достигают меньшего <<эллипса>> и остаются в нем никогда не покидая его.

\begin{figure}[h]
\begin{minipage}[h]{0.47\linewidth}
\center{\includegraphics[width=1\linewidth]{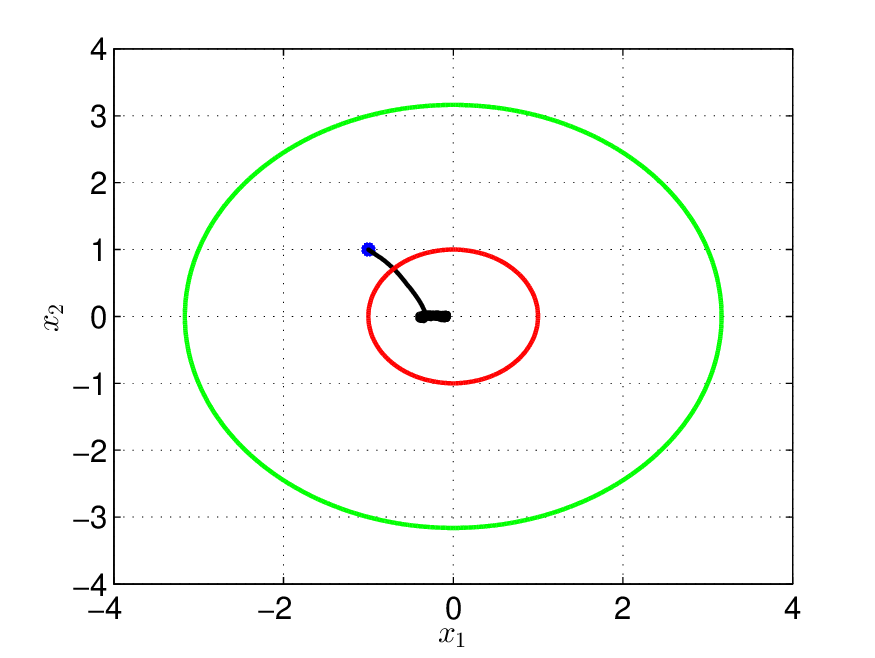}}
\end{minipage}
\hfill
\begin{minipage}[h]{0.47\linewidth}
\center{\includegraphics[width=1\linewidth]{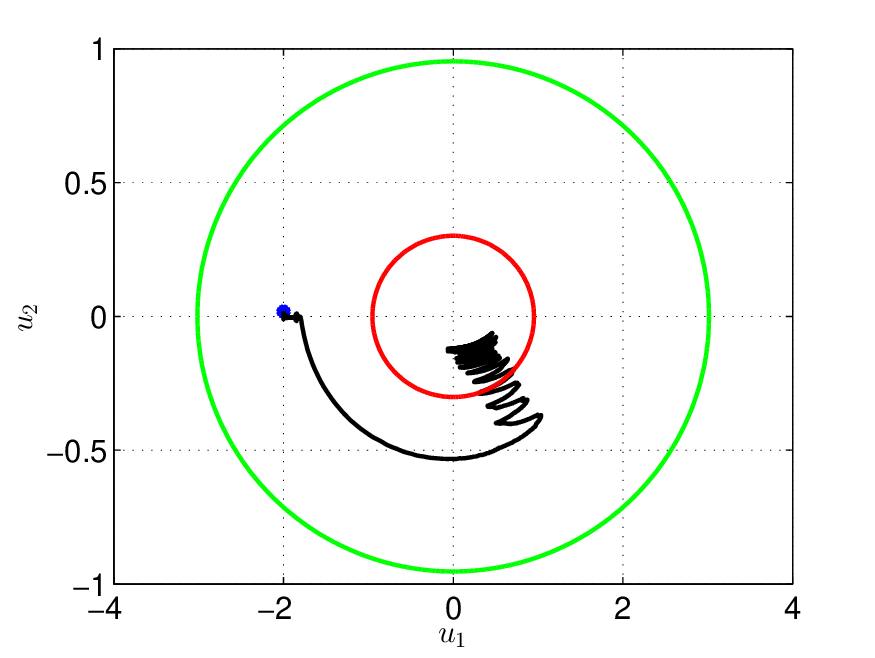}}
\end{minipage}
\caption{Фазовые траектории в замкнутой системе по $(x_1,x_2)$ и $(u_1,u_2)$.}
\label{Fig_1_Ex2}
\end{figure}

\begin{figure}[h!]
\center{\includegraphics[width=0.5\linewidth]{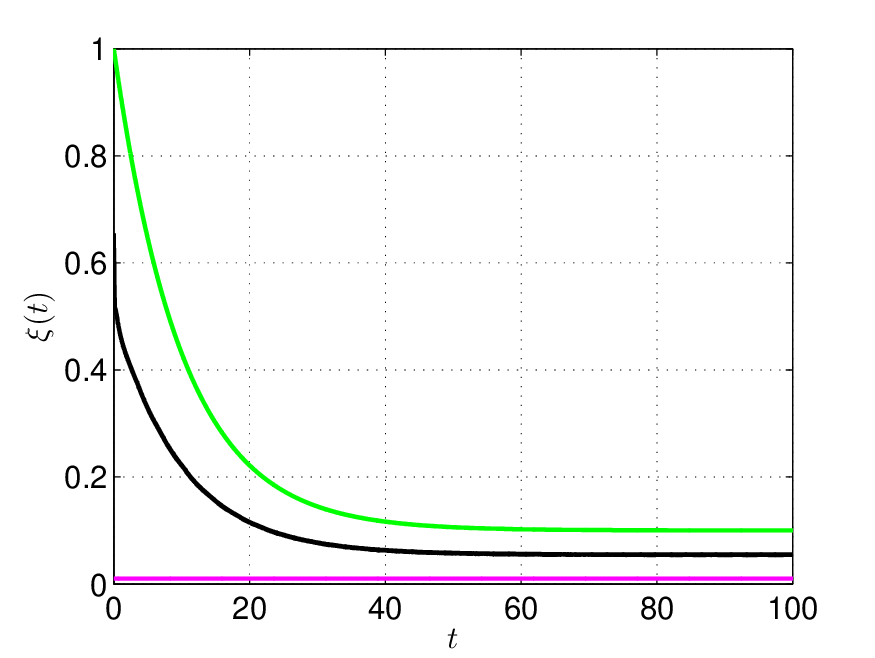}}
\caption{Переходные процессы по $\xi(t)$ в замкнутой системы.}
\label{Fig_Xi_Ex2}
\end{figure}


\section{Управление по выходу}
\label{Sec5}

Рассмотрим объект управления вида
\begin{equation}
\label{eq5_1}
\begin{array}{l}
\dot{x}=Ax+Bu+Df,
\\
y=Lx,
\end{array}
\end{equation}
где $x \in \mathbb R^n$, $u \in \mathcal U \subset \mathbb R$, $y \in \mathcal Y \subset \mathbb R$, $f \in \mathbb R$ и $|f(t)| \leq \bar{f}$ для всех $t$.
Матрицы $A$, $B$, $D$ и $L$ имеют соответствующие размерности. Объект управления \eqref{eq5_1} строго минимально-фазовый \cite{Miroshnik00}. 
Множества $\mathcal Y$ и $\mathcal U$ заданы в виде
\begin{equation}
\label{eq_XU_5}
\begin{array}{l}
\mathcal Y = \left\{x \in \mathbb R:~ p_1 y^2 \leq l_y(t) \right\},
~~~
\mathcal U = \left\{u \in \mathbb R:~ p_0 u^2 \leq l_u(t) \right\},
\end{array}
\end{equation}
где $p_1>0$, $p_0>0$, $l_y(t) > 0$ и $l_u(t)>0$ задаются разработчиком.

Представим сигнал управления в виде суммы \eqref{eq_C_L} и определим переменную $\xi$ в \eqref{eq2_2} как
\begin{equation}
\label{eq5_3}
\xi=p_1 y^2 + p_2 u_1^2 + p_{3}(|u_2|+\delta)^2,
\end{equation}
где $\delta$, $p_2$ и $p_3$ определены в разделе \ref{Sec4}. 
Выбрав $\overline{g}(t)=\inf\{l_y(t),l_u(t)\}$ получим, что множество $\mathcal{S}$ в \eqref{eq3_3_G} включает в себя ограничения \eqref{eq_XU_5}.
Цель управления состоит в выполнении целевого условия \eqref{eq3_3_G}. 

Метод \cite{Furtat21} для объекта \eqref{eq5_1} гарантирует нахождение только выходного сигнала в заданном множестве. 
В отличие от \cite{Furtat21}, предложенный метод позволит обеспечить нахождение выходного и управляющего сигналов в заданном множестве.

Преобразуем \eqref{eq5_1} к виду
\begin{equation}
\label{eq5_2}
\begin{array}{l}
Q(p)y(t)=R(p)u(t)+\phi(t).
\end{array}
\end{equation}
Здесь $Q(p)=\det(pI-A)$, 
$R(p)=L(pI-A)^*B$, 
$(pI-A)^*$ -- присоединенная матрица, 
$\phi(t)=L(pI-A)^*\big[x(0) + D f(t) + Bu(0) + B f(0) \big]$ и $p=d/dt$.

%
%
Сформулируем основной результат раздела \ref{Sec5}.

\begin{theorem}
\label{Th03}
Пусть для преобразования \eqref{eq2_2} выполнены условия (а)-(г), $\frac{\partial\Phi(\varepsilon,t)}{\partial \varepsilon}>0$  для любых $\varepsilon$ и $t$. 
 Если для заданных $\delta>0$, $\mu>0$, $p_1>0$, $p_2>0$, $p_3>0$, $\beta>0$, $c>0$ и $k \in \mathbb R$ существуют $H>0$, $\alpha>0$, $\tau_i>0$, $i=1,...,5$ такие, что при $v=\pm \bar{f}$ разрешимы следующие линейные неравенства:
\begin{equation}
\label{eq_LMI2}
\begin{array}{l}
\begin{bmatrix}
-\alpha + 0.5 \tau_1  & 0.5 v \mu^{-1} & -0.5 \\
* & -\tau_2 & 0 \\
* & * & -\tau_3
\end{bmatrix} \leq 0,
\\
c \tau_1 \geq \hat{\phi}^2 \tau_2 + \gamma^2 \tau_3,
\end{array}
\end{equation}
\begin{equation}
\label{eq_LMI3}
\begin{array}{l}
\begin{bmatrix}
\bar{A}^{\rm T}H+Q\bar{A}+\beta H  & H B & H D \\
* & -\tau_4 & 0 \\
* & * & -\tau_5
\end{bmatrix} \leq 0,
\\
H \geq \bar{p}_1 L^{\rm T}L,
\\
\frac{\inf\{\bar{g}(t)\}}{\bar{p}_1} \beta \geq \frac{\inf\{\bar{g}(t)\}}{p_3} \tau_4 + \hat{\phi}^2 \tau_5.
\end{array}
\end{equation}
Тогда закон управления \eqref{eq_C_L} обеспечивает выполнение целевого условия \eqref{eq3_3_G}, где
\begin{equation}
\label{eq5_5}
\begin{array}{l}
u_1=ky,
\\
\dot{u}_2 =- \frac{2}{p_3 (|u_2|+\delta)} \textup{sign}(u_2) 
\left[
\alpha \varepsilon 
+2 \bar{p}_1 y \frac{p R(p)}{\bar{Q}(p)}u_2 + 
\mu \bar{p}_1^2 \textup{sign}(\varepsilon) y^2
\right],
\end{array}
\end{equation}
$\bar{A}=A+kBL$, $\bar{p}_1=p_1+k^2 p_2$.
\end{theorem}

\begin{proof}
С учетом $u_1$ в \eqref{eq5_5}, перепишем \eqref{eq5_3} и \eqref{eq5_2} в виде
\begin{equation}
\label{eq5_new}
\begin{array}{l}
\bar{Q}(p)y(t)=R(p)u_2(t)+\phi(t),
\\
\xi=\bar{p}_1 y^{2} + p_3(|u_2|+\delta)^2.
\end{array}
\end{equation}
Здесь $\bar{Q}(p)=Q(p)-kR(p)$. Подставив \eqref{eq5_2} в \eqref{eq3_3}, получим
\begin{equation}
\label{eq5_4}
\begin{array}{l}
\dot{\varepsilon}=\left(\frac{\partial\Phi(\varepsilon,t)}{\partial \varepsilon}\right)^{-1}
\Big[
2 \bar{p}_1 y \frac{p R(p)}{\bar{Q}(p)}u_2 + 2 \bar{p}_1 y \bar{\phi}(t)
+ 2 p_3 (|u_2| + \delta) \textup{sign}(u_2) \dot{u}_2 - \frac{\partial\Phi(\varepsilon,t)}{\partial t}
\Big],
\end{array}
\end{equation}
где $\bar{\phi}(t)=\frac{p}{\bar{Q}(p)} \phi(t)$ -- ограниченная функция.
 Обозначим $\hat{\phi} = \sup \{\bar{\phi}(t)\}$.

Для анализа устойчивости \eqref{eq5_4} рассмотрим функцию Ляпунова \eqref{eq4_6}.
Взяв производную от \eqref{eq4_6} вдоль решений \eqref{eq5_4} и воспользовавшись неравенством Юнга, получим
\begin{equation}
\label{eq5_8}
\begin{array}{lll}
\dot{V}_1 = &
\left(\frac{\partial\Phi(\varepsilon,t)}{\partial \varepsilon}\right)^{-1}
\varepsilon
\Big[
2 \bar{p}_1 y \frac{p R(p)}{\bar{Q}(p)}u_2 + 2 \bar{p}_1 y \bar{\phi}(t)
+ 2 p_3 (|u_2| + \delta) \textup{sign}(u_2) \dot{u}_2 - \frac{\partial\Phi(\varepsilon,t)}{\partial t}
\Big]
\leq
\\
&\leq
\left(\frac{\partial\Phi(\varepsilon,t)}{\partial \varepsilon}\right)^{-1}
\varepsilon
\Big[
2 \bar{p}_1 y \frac{p R(p)}{\bar{Q}(p)}u_2 + 
\mu \bar{p}_1^2 \textup{sign}(\varepsilon) y^2 + \mu^{-1} \textup{sign}(\varepsilon) \bar{\phi}^2+
\\
&+ 2 p_3 (|u_2| + \delta) \textup{sign}(u_2) \dot{u}_2 - \frac{\partial\Phi(\varepsilon,t)}{\partial t}
\Big].
\end{array}
\end{equation}

С учетом $u_2$ из \eqref{eq5_5}, перепишем \eqref{eq5_8} как
\begin{equation}
\label{eq5_dot_V2}
\begin{array}{lll}
\dot{V}_1 = &
\left(\frac{\partial\Phi(\varepsilon,t)}{\partial \varepsilon}\right)^{-1}
\Big[
-\alpha \varepsilon^2 + \mu^{-1} \textup{sign}(\varepsilon) \varepsilon \bar{\phi}^2
- \varepsilon \frac{\partial\Phi(\varepsilon,t)}{\partial t}
\Big].
\end{array}
\end{equation}

Потребуем при $V_1 \geq c$ выполнение условия $\dot{V}_1 \leq 0$, принимая во внимание ограничения $\bar{\phi}^2 \leq \hat{\phi}^2$ и $\left( \frac{\partial\Phi(\varepsilon,t)}{\partial t} \right)^2 \leq \gamma^2$.
 Так как $\frac{\partial\Phi(\varepsilon,t)}{\partial t} > 0$ не влияет на знак выражения \eqref{eq4_6_dot_V1}, то перепишем вышеназванные условия в виде
\begin{equation}
\label{eq5_6_dot_V2}
\begin{array}{l}
- \alpha \varepsilon^2 
+ \mu^{-1} \varepsilon \textup{sign}(\varepsilon) \bar{\phi}^2
-\varepsilon \frac{\partial\Phi(\varepsilon,t)}{\partial t} 
\leq 0
~~~ \forall \left(\varepsilon, \bar{\phi}, \frac{\partial\Phi(\varepsilon,t)}{\partial t} \right):
\\
0.5 \varepsilon^2 \geq c,
~~~\bar{\phi}^2 \leq \hat{\phi}^2,
~~~\left( \frac{\partial\Phi(\varepsilon,t)}{\partial t} \right)^2 \leq \gamma^2.
\end{array}
\end{equation}

Обозначив $z=col \left \{\varepsilon,\bar{\phi},\frac{\partial\Phi(\varepsilon,t)}{\partial t} \right \}$, 
перепишем \eqref{eq5_6_dot_V2} в матричном виде:
\begin{equation}
\label{eq5_6_dot_V3}
\begin{array}{l}
z^{\rm T}
\begin{bmatrix}
-\alpha  & 0.5 \mu^{-1} \textup{sign}(\varepsilon) \bar{\phi} & -0.5 \\
* & 0 & 0 \\
* & * & 0
\end{bmatrix}
z \leq 0,
\\
z^{\rm T}
\begin{bmatrix}
-0.5  & 0 & 0 \\
* & 0 & 0 \\
* & * & 0
\end{bmatrix}
z \leq -1,
~~~
z^{\rm T}
\begin{bmatrix}
0  & 0 & 0 \\
* & 1 & 0 \\
* & * & 0
\end{bmatrix}
z \leq \bar{f}^2,
~~~

z^{\rm T}
\begin{bmatrix}
0  & 0 & 0 \\
* & 0 & 0 \\
* & * & 1
\end{bmatrix}
z \leq \gamma^2.
\end{array}
\end{equation}

Используя S-процедуру и условия разрешимости матричных неравенств с политопной неопределенностью $\textup{sign}(\varepsilon) \bar{\phi} = [-\hat{\phi},\hat{\phi}]$ \cite{Fridman10,Polyak14}, выражения \eqref{eq5_6_dot_V3} будут выполнены, если будут выполнены условия 
\eqref{eq_LMI2}.
Значит, $\varepsilon(t)$ ограниченная функция для любых $t$. 
Тогда из теоремы \ref{Th3_1} следует, что будет выполнено целевое условие \eqref{eq3_3_G}. 
Следовательно, сигналы $y$, $u_1$, $u_2$ и $u$ ограниченные.

Однако условие \eqref{eq3_3_G} может быть выполнено при $u$ стремящимся в пределе к нулю и $y$ стремящимся в пределе к границе интервала $\left(-\sqrt{\frac{\bar{g}(t)}{\bar{p}_1}};\sqrt{\frac{\bar{g}(t)}{\bar{p}_1}}\right)$. 
Несмотря на то, что это удовлетворяет поставленной цели \eqref{eq3_3_G}, на  практике такая ситуация может потребовать больших вычислительных затрат. 
Для избежания этого найдем дополнительные условия с использованием функции Ляпунова вида
\begin{equation}
\label{eq5_V2_1}
\begin{array}{l}
V_2=x^{\rm T} Hx
\end{array}
\end{equation}
и перепишем \eqref{eq5_new} как
\begin{equation}
\label{eq5_V2_1}
\begin{array}{l}
\dot{x}=\bar{A} x + Bu_2 + D\phi,
~~~
y=Lx.
\end{array}
\end{equation}

Потребуем при $V_2 \geq \frac{\inf{\bar{g}(t)}}{\bar{p}_1}$ выполнение условий $\dot{V} \leq 0$ и $H \geq \bar{p}_1 L^{\rm T}L$ (последнее означает, что эллипсоид $x^{\rm T} H x =1$ содержится внутри цилиндра $x^{\rm T} \bar{p}_1 L^{\rm T}L x =1$), принимая во внимание ограничения $u_2^2 \leq \frac{\inf{\bar{g}(t)}}{p_3}$ и $\phi^2 \leq \hat{\phi}^2$.
 Перепишем вышеназванные условия в виде
\begin{equation}
\label{eq5_V2_dot_1}
\begin{array}{l}
\dot{V}_2=
x^{\rm T}(\bar{A}^{\rm T}H+H\bar{A})x+2x^{\rm T}H B u_2+2x^{\rm T}H D \phi \leq 0
~~~ \forall (x, u_2, \phi):
\\
x^{\rm T} H x \geq \frac{\inf{\bar{g}(t)}}{\bar{p}_1},
~~~u_2^2 \leq \frac{\inf{\bar{g}(t)}}{p_3},
~~~\phi^2 \leq \hat{\phi}^2.
\end{array}
\end{equation}
Обозначив $s=col\{x, u_2, \phi \}$, перепишем \eqref{eq5_V2_dot_1} как
\begin{equation}
\label{eq5_V2_dot_2}
\begin{array}{l}
s^{\rm T}
\begin{bmatrix}
\bar{A}^{\rm T}H+H\bar{A}  & H B & H D \\
* & 0 & 0 \\
* & * & 0
\end{bmatrix}s \leq 0,
\\
- s^{\rm T}
\begin{bmatrix}
H  & 0 & 0 \\
* & -1 & 0 \\
* & * & 0
\end{bmatrix}
s \geq -\frac{\inf{\bar{g}(t)}}{\bar{p}_1},
\\
s^{\rm T}
\begin{bmatrix}
0  & 0 & 0 \\
* & -1 & 0 \\
* & * & 0
\end{bmatrix}
s \leq \frac{\inf\{\bar{g}(t)\}}{p_3},
~~
s^{\rm T}
\begin{bmatrix}
0  & 0 & 0 \\
* & 0 & 0 \\
* & * & -1
\end{bmatrix}
s \leq \hat{\phi}.
\end{array}
\end{equation}
Согласно S-процедуре \cite{Fridman10,Polyak14}, неравенства \eqref{eq5_V2_dot_2} будут выполнены, если будут выполнены условия \eqref{eq_LMI02}.
Теорема \ref{Th03} доказана.
\end{proof}


\sl{П\,р\,и\,м\,е\,р\,\,3}.
\rm
Рассмотрим неустойчивый объект \eqref{eq5_1} со следующими параметрами:
\begin{equation*}
\label{eq5_9}
\begin{array}{l}
A=
\begin{bmatrix}
0 & 1 & 0\\
0 & 0 & 1\\
3 & 5 & 1
\end{bmatrix},
~~~
B=
\begin{bmatrix}
0  \\
0 \\
1
\end{bmatrix},
~~~
D=
\begin{bmatrix}
0.1  \\
0.2 \\
1
\end{bmatrix},
~~~
L=
\begin{bmatrix}
1 & 2 & 1
\end{bmatrix},
~~~
x(0)=
1.1
\begin{bmatrix}
1  \\
1 \\
1
\end{bmatrix},
\end{array}
\end{equation*}
сигнал $f(t)$ определен в примере 2. 
Тогда $R(p)=(p+1)^2$, $Q(p)=(p+1)^3$ и $\hat{\phi}=0.22$.

Пусть в \eqref{eq_XU_5} определены следующие параметры: $p_1=0.1$, $p_0=0.01$, $l_y(t)=8$ и $l_u(t)=8$.

Зададим 
$\delta=0.01$ в \eqref{eq5_3}, $k=-4$ и $\mu=0.01$ в \eqref{eq5_5}, а также $T(\varepsilon)$ из примера 2. 
 С учетом $p_0=0.01$ имеем, что $p_{2}= 0.0101$ и $p_3=1.01$.
 Задав $\beta=0.1$ и $c=1$, неравенства \eqref{eq_LMI2} будут разрешимы, например, при $\alpha=20.2$.
 
На рис.~\ref{Fig_1_Ex3} изображены траектории по $\xi(t)$, $y(t)$, $u_1(t)$ и $u_2(t)$ при $\overline{g}=7e^{-0.1t}+1$ и $\underline{g}=4.95e^{-0.1t}+0.05$ в \eqref{eq2_3}, 
на рис.~\ref{Fig_2_Ex3} -- при $\overline{g}=3.5 \cos(0.5t)+4.5$ и $\underline{g}=\cos(0.5t)+1.05$ в \eqref{eq2_3}. Из рис.~\ref{Fig_1_Ex3} и \ref{Fig_2_Ex3} видно, что сигналы $\xi$, $y$, $u_1$ и $u_2$ никогда не покидают заданных ограничений, характер которых может быть дополнительно определен разработчиком произвольно, например, в первом случае экспоненциальными функциями, а во втором -- синусоидальными.

\begin{figure}[h]
\begin{minipage}[h]{0.49\linewidth}
\center{\includegraphics[width=1\linewidth]{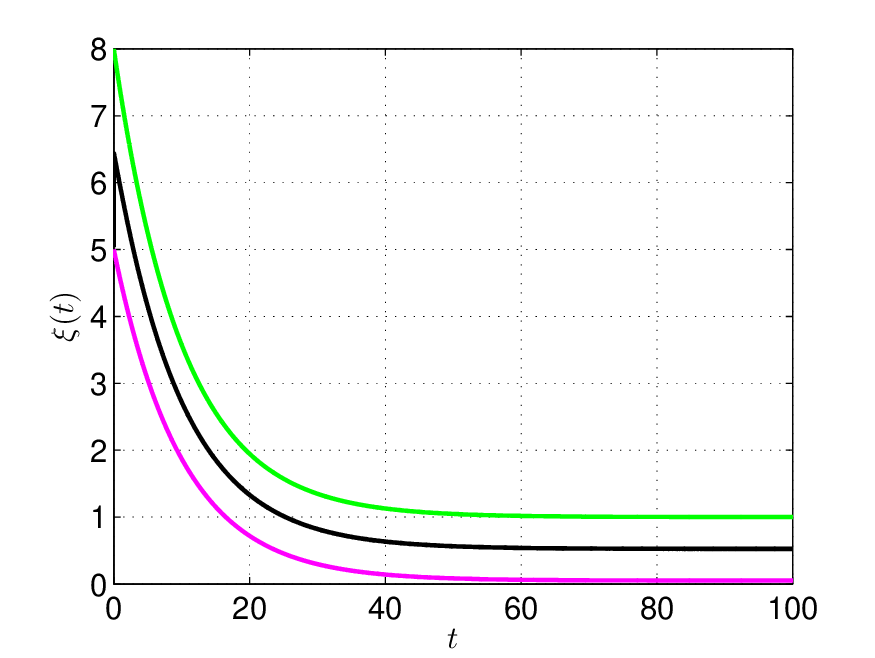}}
\end{minipage}
\hfill
\begin{minipage}[h]{0.49\linewidth}
\center{\includegraphics[width=1\linewidth]{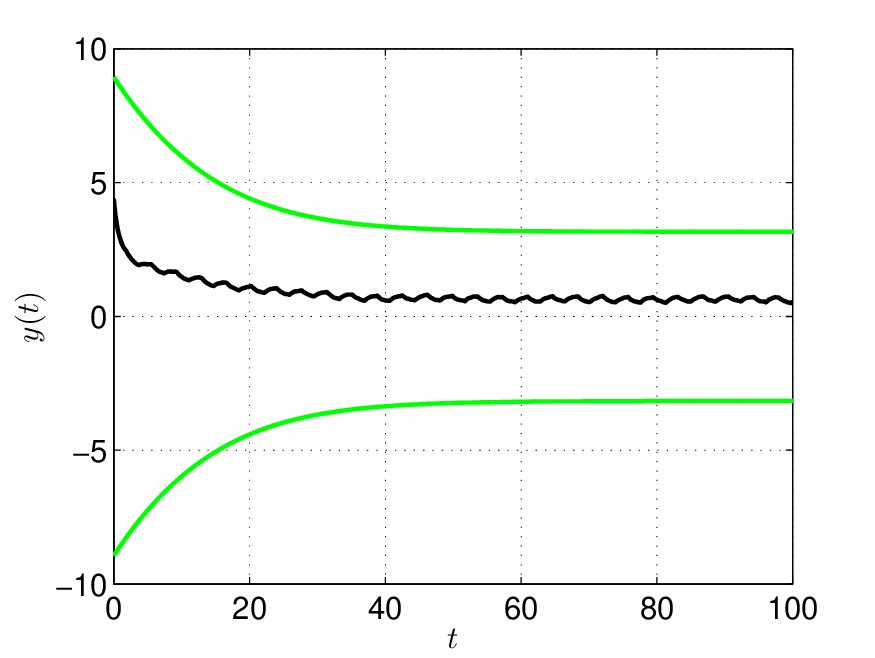}}
\end{minipage}
\vfill
\begin{minipage}[h]{0.49\linewidth}
\center{\includegraphics[width=1\linewidth]{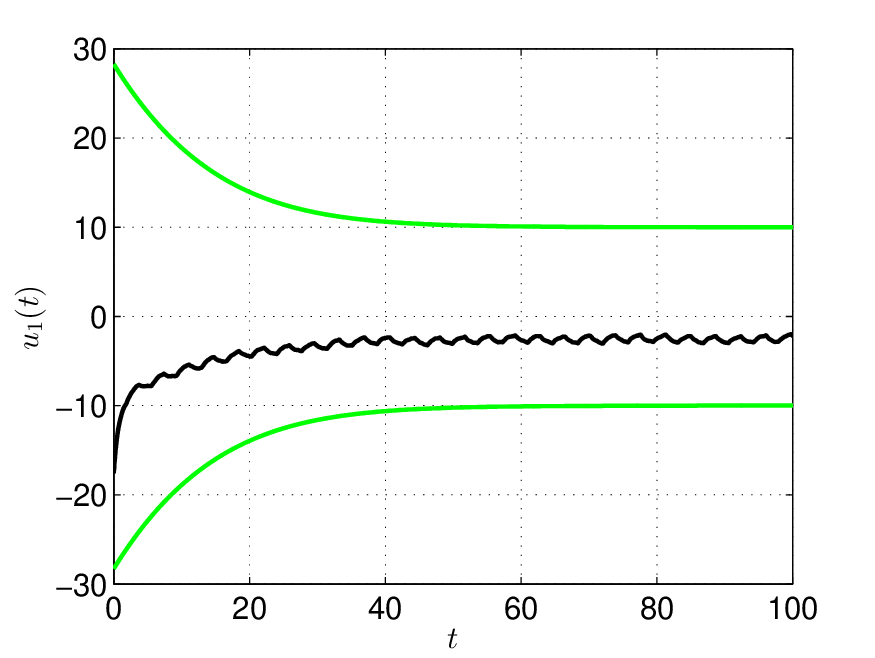}}
\end{minipage}
\hfill
\begin{minipage}[h]{0.49\linewidth}
\center{\includegraphics[width=1\linewidth]{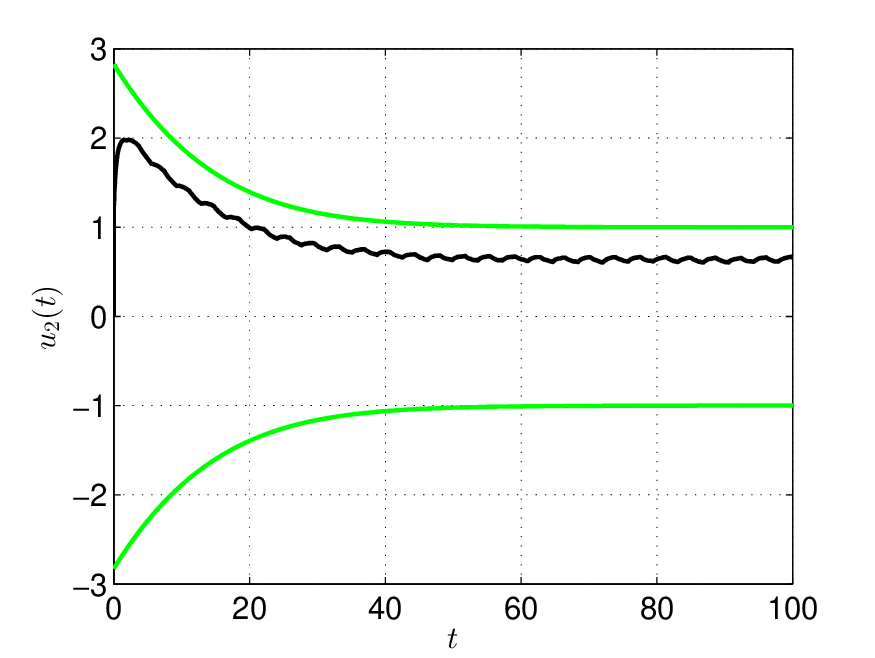}}
\end{minipage}
\caption{Переходные процессы по $\xi(t)$, $y(t)$, $u_1(t)$ и $u_2(t)$.}
\label{Fig_1_Ex3}
\end{figure}

\begin{figure}[h]
\begin{minipage}[h]{0.49\linewidth}
\center{\includegraphics[width=1\linewidth]{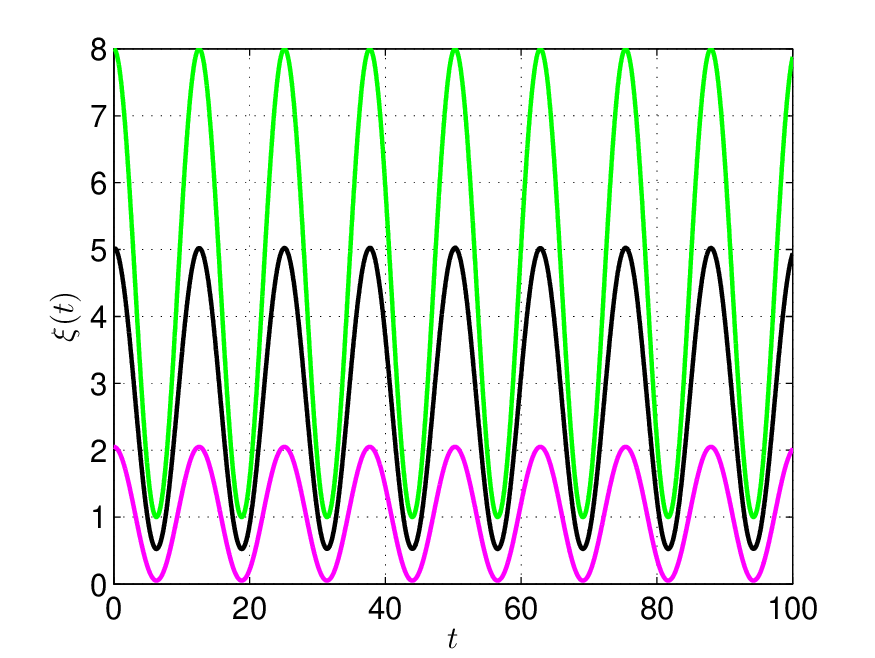}}
\end{minipage}
\hfill
\begin{minipage}[h]{0.49\linewidth}
\center{\includegraphics[width=1\linewidth]{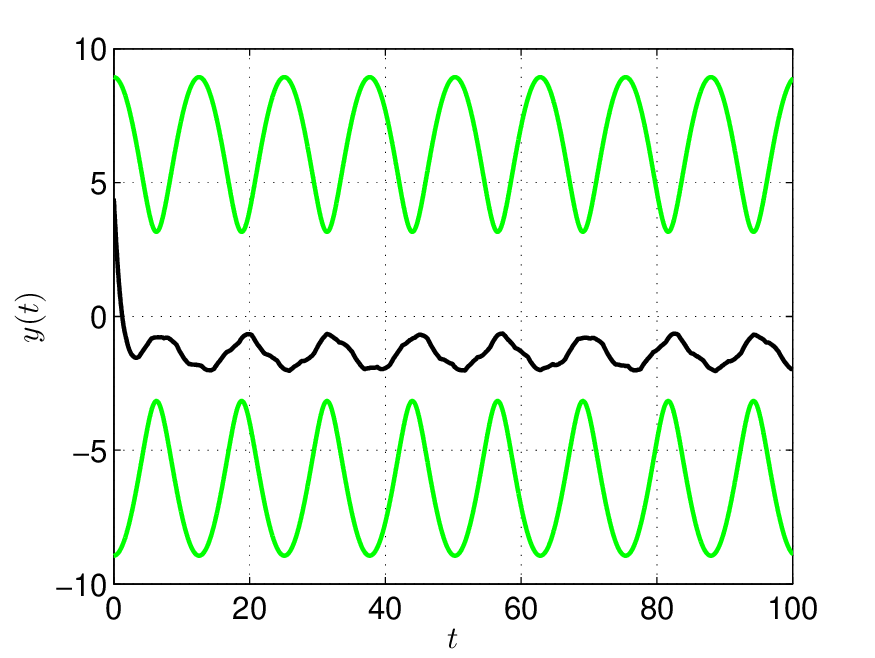}}
\end{minipage}
\vfill
\begin{minipage}[h]{0.49\linewidth}
\center{\includegraphics[width=1\linewidth]{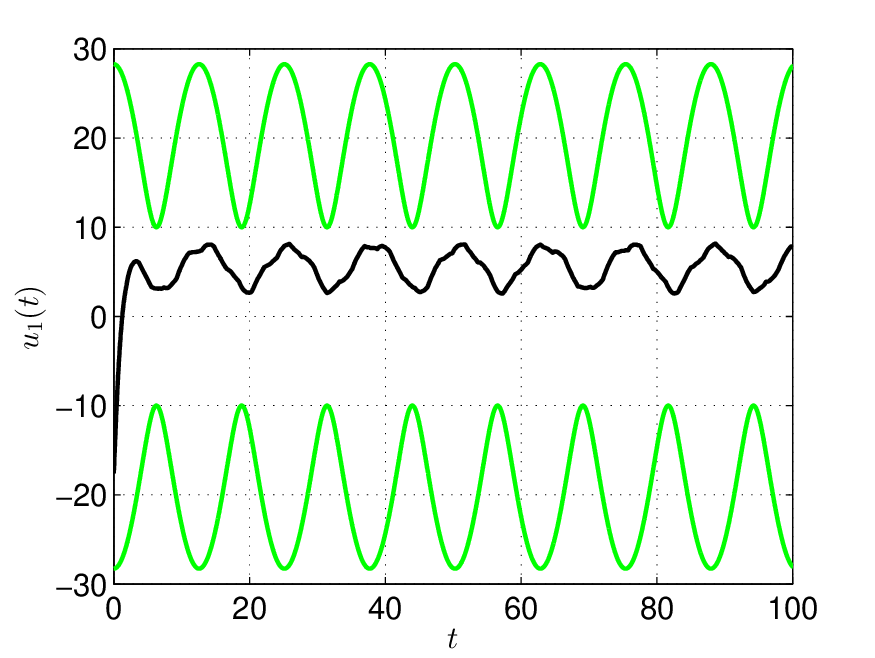}}
\end{minipage}
\hfill
\begin{minipage}[h]{0.49\linewidth}
\center{\includegraphics[width=1\linewidth]{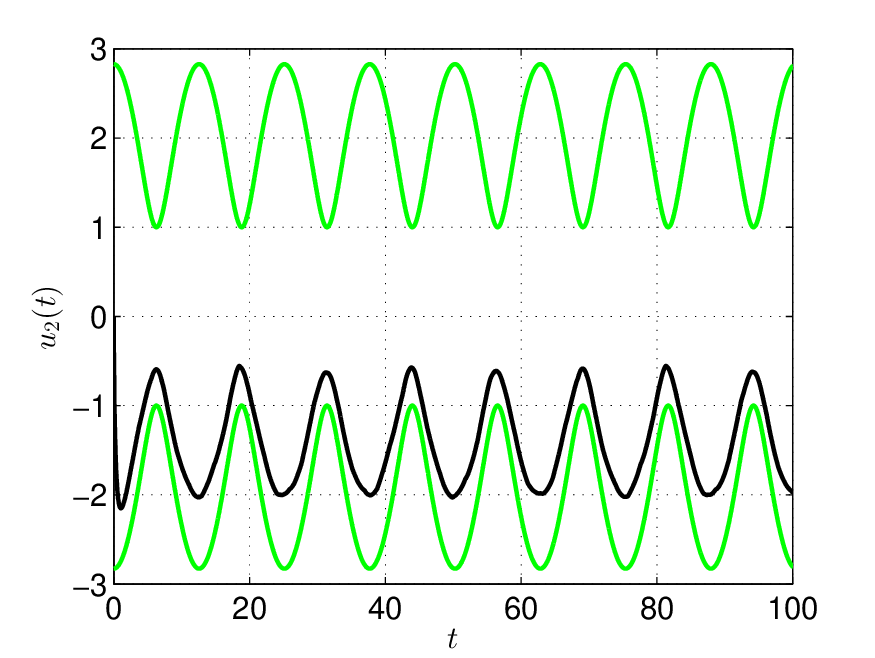}}
\end{minipage}
\caption{Переходные процессы по $\xi(t)$, $y(t)$, $u_1(t)$ и $u_2(t)$.}
\label{Fig_2_Ex3}
\end{figure}

\section{Заключение}
В статье предложено развитие метода \cite{Furtat21} на динамические системы с произвольным соотношением числа управлений и выходных сигналов и гарантией их нахождения в заданных множествах. 
Разработанный метод применяется для решения задач управления по состоянию и по выходу линейными системами с учетом ограничений на сигнал управления и выходные переменные, где размерность регулируемых переменных больше размерности сигнала правления. 
Также, в отличие от \cite{Furtat21}, устойчивость замкнутой системы и синтез параметров регулятора сформулированы в терминах разрешимости линейных матричных неравенств.
Результаты моделирования подтвердили теоретические выводы и показали эффективность предложенного метода при наличие параметрической неопределенности и возмущений.

~\

~\

%

\end{document}